\documentclass[12pt,reqno]{amsart}
\usepackage[letterpaper, total={6in, 8in}, left=1.25in, top=1.5in]{geometry}
\usepackage{amsmath}
\usepackage{amsfonts}
\usepackage{amssymb}
\usepackage{graphicx}
\usepackage{color}
\usepackage{hyperref}
\newtheorem{theorem}{Theorem}[section]

\newtheorem*{theorem*}{Theorem}
\newtheorem{lemma}{Lemma}[section]

\newtheorem{proposition}{Proposition}[section]

\def\a{\alpha}
\def\l{\lambda}

\def\p{\partial}

\def\R{\mathbb{R}}

\def\Rm{\operatorname{Rm}}

\numberwithin{equation}{section}

\begin{document}
\title[The second Robin eigenvalue in non-compact ROSS]{The second Robin eigenvalue in non-compact rank-1 symmetric spaces}
\author{Xiaolong Li}
\address{Department of Mathematics, Statistics and Physics, Wichita State University, Wichita, KS 67260, USA}
\email{xiaolong.li@wichita.edu}

\author{Kui Wang} 
\address{School of Mathematical Sciences, Soochow University, Suzhou, 215006, China}
\email{kuiwang@suda.edu.cn}

\author{Haotian Wu}
\address{School of Mathematics and Statistics, The University of Sydney, NSW 2006, Australia}
\email{haotian.wu@sydney.edu.au}


\begin{abstract}
In this paper, we prove a quantitative spectral inequality for the second Robin eigenvalue in non-compact rank-1 symmetric spaces. In particular, this shows that for bounded domains in non-compact rank-1 symmetric spaces,  the geodesic ball maximises the second Robin eigenvalue among domains of the same volume, with negative Robin parameter in the regime connecting the first nontrivial Neumann and Steklov eigenvalues. This result generalises the work of Freitas and Laugesen in the Euclidean setting \cite{FL21} as well as our previous work in the hyperbolic space \cite{LWW20}.
\end{abstract}

\maketitle

\section{Introduction}
The Robin eigenvalue problem is to solve 
\begin{align}\label{eq 1.1}
\begin{cases} 
-\Delta u =\l u  & \text{ in } \Omega, \\
\frac{\p u}{\p \nu} +\a u  =0 & \text{ on } \p \Omega,
\end{cases}
\end{align}
where $\Omega$ is a bounded domain with Lipschitz boundary, $\Delta$ denotes the Laplace-Beltrami operator, $\nu$ denotes the outward unit normal to $\p \Omega$, and $\alpha \in \R$ is the Robin parameter. The eigenvalues, denoted by $\l_{k,\a}( \Omega)$ for $k=1, 2, \cdots$, are increasing and continuous in $\alpha$, and for each $\alpha$ satisfy
\begin{align*} 
\l_{1, \a}(\Omega)\le \l_{2,\a}(\Omega) \le \l_{3,\a}(\Omega) \le \cdots \rightarrow \infty,
\end{align*}
where each eigenvalue is repeated according to its multiplicity. The first eigenvalue is simple if $\Omega$ is connected; the first eigenfunction is positive. The Robin eigenvalue problem  generates a global picture of the spectrum of the Laplace operator. Indeed, the Neumann $(\a =0)$, the Steklov ($\l =0$ and  $-\a$ equals the Steklov eigenvalue) and the Dirichlet ($\a \to +\infty)$ eigenvalue problems are all special cases of the Robin eigenvalue problem. Hence, existing results on the Dirichlet, Neumann or Steklov eigenvalues naturally motivate the investigation on the Robin eigenvalues.

In Euclidean space, the classical Faber-Krahn inequality asserts that the ball uniquely minimizes the first Dirichlet eigenvalue among bounded domains with the same volume. When $\alpha>0$, the ball is also the unique minimizer of the first Robin eigenvalue $\l_{1,\a}(\Omega)$ among domains of the same volume in $\mathbb{R}^n$, as was shown in dimension two by Bossel \cite{Bos86} in 1986 and extended to all dimensions $n\geq 2$ by Daners \cite{Dan06} in 2006. An alternative approach via the calculus of variations was found by Bucur and Giacomini \cite{BG10, BG15} later. For negative values of $\a$, it was conjectured by Bareket \cite{Bar77} in 1977 that the ball would be the maximizer among domains in $\mathbb{R}^n$ with the same volume. However, in 2015, Freitas and  Krej\v{c}i\v{r}{i}k \cite{FK15} disproved Bareket's conjecture by showing that the ball is not a maximizer for sufficiently negative values of $\a$. In the same paper, the authors showed that in dimension two, the disk uniquely maximizes $\l_{1,\a}(\Omega)$ for $\a <0$ with $|\a|$ sufficiently small, and conjectured that the maximizer still has radial symmetry whenever $\a<0$ and should switch from a ball to a shell at some critical value of $\a$. 

Let us turn to the shape optimization problem for the second Robin eigenvalue $\l_{2,\a}(\Omega)$. Suppose for the moment that $\Omega\subset\mathbb{R}^n$. When $\a >0$, both the second Dirichlet eigenvalue and the second Robin eigenvalue are uniquely minimized by the disjoint union of two equal balls among bounded Lipschitz domains of the same volume. This was proved by Kennedy \cite{James09}. When $\a =0$, we have $\l_{2,0}(\Omega)=\mu_1(\Omega)$, the first nonzero Neumann eigenvalue, for which the classical Szeg\"o-Weinberger inequality \cite{Sz54,Wei56} states that among domains with the same volume, the ball uniquely maximizes $\mu_1(\Omega)$, see \cite{BP12}  for stability version as well. When $\a<0$, it is expected, cf. \cite[Problem 4.41]{He17}, that $\l_{2,\a}(\Omega)$ should be maximal on the ball for a range of Robin parameters. This expectation has recently been confirmed by Freitas and Laugesen via Szeg\"o way \cite{FL20} and Weineberger way \cite{FL21} respectively. Precisely the following theorem holds true.
\begin{theorem}[Theorem A of \cite{FL21}]
Let $\Omega\subset \R^n$ be a bounded  Lipschitz domain, $n\ge 2$ and  $B$ be a round ball of the same volume as $\Omega$.  If $\a\in[-\frac{n+1}{n}R,0]$, then
\begin{align}\label{eq:1-2}
\l_{2,\a}(\Omega) \le \l_{2,\a}(B),
\end{align}
where $R$ is the radius of $B$.
Equality holds if and only if $\Omega$ is a round ball.  
\end{theorem}
Inequality \eqref{eq:1-2} asserts that the ball uniquely maximizes $\l_{2,\a}(\Omega)$ among domains of the same volume provided that $\a$ lies in a regime connecting the first nonzero Neumann eigenvalue $\mu_1$ and the first nonzero Steklov eigenvalue $\sigma_1$, namely $\a \in [-\frac{n+1}{n}R^{-1}, 0]$. Taking $\a=0$ and $\a=-1/R$ recovers the Szeg\"o-Weinberger inequality \cite{Wei54} for $\mu_1(\Omega)$ and the Brock-Weinstock inequality \cite{Br01} for $\sigma_1(\Omega)$ respectively, and both classical inequalities assert that the ball is the unique maximizer among domains with the same volume in Euclidean space. We mention that the stability version of the Brock-Weinstock inequality has been proved in \cite{BPR12}. 

It is known that  the Szeg\"o-Weinberger inequality holds for domains in the hemisphere, in the hyperbolic space \cite{AB95} and in rank-1 symmetric spaces (ROSS) \cite{AS96}, asserting that the geodesic ball maximises the first nonzero Neumann eigenvalues among domains of the same volume. Binoy and Santhanam proved \cite{BS14} the Brock-Weinstock inequality in hyperbolic space and non-compact ROSS, and Castillon and Ruffini  proved \cite{CR19} a stability version of the Brock-Weinstock inequality in non-compact ROSS and that the Brock-Weinstock inequality does not hold on sphere in general. Recently, in their work \cite{LWW20}, the authors have extended the result of Freitas and Laugesen \cite{FL21} to bounded domains in real hyperbolic spaces, and proved  that in complete simply connected nonpositively curved space forms, geodesic balls uniquely maximize the second Robin eigenvalue among domains with the same volume, with
negative Robin parameter in the regime connecting the first nontrivial Neumann and Steklov eigenvalues. 

In this paper, we continue our study of the optimization problem for the second Robin eigenvalue for some negative Robin parameter \cite{LWW20}, and prove the following quantitative
Freitas-Laugesen  inequality in non-compact ROSS.
\begin{theorem}\label{thm}
Let $M^m = \mathbb{K} H^n$ be a non-compact ROSS of real dimension $m=kn$ (cf. Section \ref{sect.2}). For any $v>0$, there exists a constant $C=C(m, k, v)$ such that for any domain $\Omega\subset M$ with $|\Omega|=v$, any geodesic ball $B$ with $|B|=v$, and any $\a\in [-\sigma_1(B), 0]$, we have
\begin{align}\label{1.3}
    \l_{2,\a}(\Omega)+C \mathcal{A}^2(\Omega)\le \l_{2, \a}(B),
\end{align}
where $\mathcal{A}(\Omega)$ is the Fraenkel asymmetry of $\Omega$ defined by $$\mathcal{A}(\Omega)=\inf\left\{\frac{|\Omega\setminus B|+|B\setminus\Omega|}{|B|}~:~B~ \text{is any geodesic ball},~|\Omega|=|B| \right\}.$$
Equality holds if and only if $\Omega$ is a geodesic ball.
\end{theorem}
By taking $\a=0$, Theorem \ref{thm} yields that geodesic balls uniquely maximize the first nonzero Neumann eigenvalue among domains of the same volume in noncomapct ROSS, recovering Theorem 2 of \cite{AS96} proved by Aithal and Santhanam.
By taking $\a=-\sigma_1(B)$, Theorem \ref{thm} implies that geodesic balls uniquely maximize the first nonzero Steklov eigenvalue among domains of the same volume in noncomapct ROSS, which recovers the result of Binoy and Santhanam  in \cite{BS14}, and the result of  Castillon and Ruffini in \cite{CR19}.  Noticing from \cite{CR19} that the Brock-Weinstock inequality does not hold on sphere in general,   so we mention here that Theorem \ref{thm} does not hold in compact space forms or compact ROSS. 

This paper is organized as follows. In Section \ref{sect.2}, we collect some known facts on non-comapct ROSS. In Section \ref{sect.3}, we prove some properties of second Robin eigenvalues and eigenfuctions for geodesic balls in noncomapct ROSS. In Section \ref{sect.4}, we prove a monotonicity lemma, which is useful in the proof of Theorem \ref{thm}. Section \ref{sect.5} is devoted to the proof of Theorem \ref{thm}.

\subsection*{Acknowledgements}{X. Li is partially supported by a start-up grant at Wichita State University; K. Wang is partially supported by NSFC No.11601359; H. Wu is partially supported by ARC Grant DE180101348.}

\section{Geometry of non-compact rank-1 symmetric spaces }\label{sect.2}
Let $\mathbb{K}$ denote one of the following: the field $\mathbf{R}$ of real numbers,
the field $\mathbf{C}$ of complex numbers, and the algebra $\mathbf{H}$ of quaternions or or the algebra $\mathbf{Ca}$ of octonions. Let $M^m=\mathbb{K}H^n$, a non-compact ROSS of dimension $m=kn$, where $k=\operatorname{dim}_{\mathbb{R}}(\mathbb{K})$. It is known that $M$ carries $k-1$ orthogonal complex structures $J_1,\cdots, J_{k-1}$, and for any unit vector $X,Y\in T_xM$ with $Y$ orthogonal to $X, J_1 X,\cdots, J_{k-1}X$ we have
\begin{align}\label{sect}
   \Rm(X, J_i X)X=-4 J_i X, \text{\quad and \quad}
\Rm(X, Y)X=-Y, 
\end{align}
where $1\le i\le k-1$ and $\Rm$ is the Riemannian curvature operator. 

Now we collect some known facts about geodesic polar coordinates, see for example \cite[Sect. 3]{AS96}. In geodesic polar coordinates centered at a point $p\in M$, the Riemannian density function $J(r)$ is given by
$$
J(r)=\sinh^{m-1}r\cosh^{k-1}r,
$$
where $r$ is the distance function to $p$. The Laplace operator $\Delta_M$ is given by
$$
\Delta_M=\frac{\partial^2}{\partial r^2}+H(r)\frac{\partial}{\partial r}+\Delta_{S_r},
$$
where 
$$
H(r)=(\log J(r))'=(m-1)\coth r+(k-1) \tanh r,
$$
the mean curvature of the distance sphere $S_r:=\{x\in M: \operatorname{dist}(p,x)=r\}$, 
and $\Delta_{S_r}$ denotes the Laplacian of $S_r$.  Moreover,
 the first non-zero eigenvalue of $\Delta_{S_r}$ is
 \begin{align}\label{2.2}
 \l_1(S_r)=\frac{m-1}{\sinh^2 r}-\frac{k-1}{\cosh^2 r}=-H'(r),
 \end{align}
and the associated eigenfunctions are  the linear coordinate functions restricted to $\mathbb{S}^{m-1}$, denoted by  $\psi_i(\theta)$ ($1\le i\le m$), satisfying 
\begin{align}\label{sumpsi}
 \sum_{i=1}^{m} |\nabla^{S_r} \psi_i|^2=\l_1(S_r)=-H'(r),
\end{align}
see Lemma 4.11 of \cite{CR19} for details.

\section{Robin eigenvalues for geodesic balls}\label{sect.3}
The theory of self-adjoint operators yields variational characterizations of Laplacian eigenvalues. In particular, the first two Robin eigenvalues are characterized by
\begin{align}\label{eq 1.2}
\l_{1,\a}(\Omega) = \inf \left\{ \frac{\int_{\Omega} |\nabla u|^2 \,d\mu_g +\a \int_{\p \Omega} u^2\,dA_g}{\int_{\Omega}u^2 \,d\mu_g} :  u \in W^{1,2}(\Omega)\setminus \{0\} \right\}
\end{align} 
and 
\begin{align}\label{eq 1.3}
\l_{2,\a}(\Omega) = \inf \left\{ \frac{\int_{\Omega} |\nabla u|^2 \,d\mu_g  +\a \int_{\p \Omega} u^2 \,dA_g}{\int_{\Omega}u^2 \,d\mu_g} :  u \in W^{1,2}(\Omega)\setminus \{0\},\; \int_{\Omega} u u_1\,d\mu_g =0 \right\},
\end{align}
where $u_1$ is the first eigenfunction associated with $\l_{1,\a}(\Omega)$, and  the first nonzero Steklov eigenvalue $\sigma_1(\Omega)$ is characterized variationally by
\begin{align}\label{var-stek}
\sigma_1(\Omega)=\inf\left\{ \frac{\int_\Omega |\nabla u|^2 \, d\mu_g}{\int_{\partial \Omega}  u^2 \, dA_g} : u\in W^{1,2}(\Omega)\setminus\{0\},\; \int_{\partial \Omega} u\, dA_g=0 \right\},
\end{align}
 where $d\mu_g$ is the Riemnnian measure induced by the metric $g$ and $dA_g$ is the induced measure on $\p \Omega$.

Let $M$ be a non-compact ROSS and  $B_R\subset M$ be a geodesic ball of radius $R$, and we consider the following Robin eigenvalue problem
\begin{align}\label{3.1}
\begin{cases} 
-\Delta u(x) =\l u(x),  & x\in B_R, \\
\frac{\p u}{\p \nu} +\a u  =0, & x\in \p B_R.
\end{cases}
\end{align}
Since $\l_{1, \a}$ is simple and $B_R$ is rotational symmetric, then the first eigenfunction is radial, hence $\l_{1, \a}(B_R)$ is the first eigenvalue of 
\begin{align}\label{3.2}
    -f''(r)-H(r) f'(r)=\tau f(r), \quad 0<r<R,
\end{align}
with $f'(0)=0$ and $f'(R)+\a f(R)=0$. By using the separation of variables technique,
we see that the second eigenvalue $\l_{2, \a}(B_R)$ of Problem \eqref{3.1} is either the second eigenvalue of \eqref{3.2} or the first eigenvalue of 
\begin{align}\label{3.3}
 -g''(r)-H(r) g'(r)+\l_1(S_r) g(r)=\mu g(r), \quad 0<r<R,   
\end{align}
with $g(0)=0$ and $g'(R)+\a g(R)=0$, where $\l_1(S_r)$ is given by \eqref{2.2}. Moreover it's easily seen that
the first eigenvalue of \eqref{3.3}
is characterized variationally by
\begin{align}\label{3.4}
 \mu_1= \inf_{ g\in W^{1,2}((0,R))} \left\{
    \frac{\int_0^R\left( g'(r)^2+\l_1(S_r)g^2\right) J(r) \, dr +\a  g^2(R)J(R)}{\int_0^R g^2 J(r)\,  dr}: g(0)=0\right\},
\end{align}
and the associated eigenfunction $g(r)$ can be so chosen that $g(r)>0$ for $r\in (0, R]$ and $g'(0)=1$.

We  prove the following properties of the second eigenvalue and eigenfunctions of \eqref{3.1} for negative Robin parameter.
\begin{proposition}\label{pro1}
Suppose $\a\le 0$, then the second Robin eigenfunctions of \eqref{3.1} are given by 
\begin{align*}
u_i(x)=g(r) \psi_i(\theta),\quad  i=1,2, \cdots, m,
\end{align*}
where $\psi_i(\theta)$'s are the linear coordinate functions restricted to $\mathbb{S}^{m-1}$, and $g(r):[0, R]\rightarrow [0,\infty)$ solves  
\begin{equation}\label{odeg}
g''(r)+H(r) g'(r)+\left(\l_{2,\a}(B_R)-\l_1(S_r)\right) g(r)=0
\end{equation}
with boundary condition $g(0)=0$ and $g'(R)=-\a g(R)$.
\end{proposition}

\begin{proof}
Let $\mu_1$ and $\tau_2$ be the first nonzero eigenvalue of \eqref{3.3} and the second eigenvalue of \eqref{3.2} respectively. Then it suffices to show
$
\mu_1<\tau_2.
$
 Let $g$ and $f$ be eigenfunctions associated to $\mu_1$ and $\tau_2$ resp., and  set
$$h(r):=\int_0^r g(t)\, dt-\frac 1 {\mu_1} g'(0),$$
then we deduce from \eqref{3.3} that
\begin{align}\label{3.7}
      -h''(r)-H(r) h'(r)=\mu_1 h(r), 
\end{align}
for $0<r\le R$,
particularly
$$
\mu_1 h(R)=-g'(R)-H(R) g(R)=\big(\a-H(R)\big)g(R)< 0,
$$
where in the inequality we used the assumption $\a<0$. 
Since $h'(r)=g(r)>0$ in $(0, R)$, then $h$ is negative in $(0, R]$.

Using \eqref{3.2} and \eqref{3.7}, we calculate 
\begin{align} \label{3.8}
    \Big(J\big(fh'-f'h\big)\Big)'=(\tau_2-\mu_1) J f h.
\end{align}
Recall that $f$ is an eigenfuction corresponding to the second eigenvalue of \eqref{3.2}, it must change sign in $(0,R)$. Without loss of generality, we  assume $f(r)$ is positive in $(0, a)$ for some  $a<R$ and $f(a)=0$. Integrating \eqref{3.8} over $(0,a)$ yields
\begin{align}\label{3.9}
\begin{split}
    &(\tau_2-\mu_1) \int_0^a J(t) f(t)h(t)\, dt\\
    =& J(r)\Big(f(r)h'(r)-f'(r)h(r)\Big)\Big|_0^a\\
    =& -J(a) f'(a) h(a).
    \end{split}
\end{align}
Since $f'(a)<0$, and $f(r)>0$ in $(0,a)$, and $h(r)<0$, we then get from \eqref{3.9} that
$
\mu_1<\tau_2,
$ proving the proposition.
\end{proof}

\begin{proposition}\label{prop 3.1}
Let $\sigma_1(B_R)$ be the first nonzero Steklov eigenvalue on $B_R$. Suppose Robin parameter $\a\in[-\sigma_1(B_R), 0)$. Then
\begin{enumerate}
\item for the function $g(r)$, characterized in  Proposition \ref{pro1}, we have
\begin{align}
   g'(r)>0, \text{\quad and\quad}  \frac{g'(r)}{g(r)}\ge -\a
\end{align}
for $r\in(0,R]$;
\item the second Robin eigenvalue 
$$ \l_{2,\a}(B_R)\ge 0.$$
\end{enumerate}
\end{proposition}
\begin{proof}
Recall from \eqref{sect} and \eqref{2.2} that  $\operatorname{Sect}_M\in \{-1, -4\}$ and
$$
 \l_1(S_R)=\frac{m-1}{\sinh^2 R}-\frac{k-1}{\cosh^2 R}<\frac{m-1}{R^2}=\l_1(\bar{S}_R), 
$$
where $\bar S_R$ is the boundary of the round ball  $\bar B_R$  of radius $R$  in $\R^m$, then Escobar's comparison \cite[Theorem 2]{Escobar00} gives
\begin{align}\label{sig}
\sigma_1(B_R)\le \sigma_1(\bar B_R)=\frac 1 R.
\end{align}

\emph {Proof of (1)}. Let $N(r)=J(r)g'(r)$. Using equation \eqref{odeg}, we deduce that
\begin{align*}
N'(r)=\left(\l_1(S_r)-\l_{2,\a}(B_R)\right)J(r) g(r).
\end{align*}
Recall
$$
\l_1(S_r)=\frac{m-k}{\sinh^2 r}+\frac{k-1}{\sinh^2 r\cosh^2 r}
$$
increasing in $r$, then we have that $N'(r)$ has at most one zero in $(0, R]$ and is positive near $0$. Since $N(0)=0$ and $N(R)=-\a J(R)g'(R)>0$, then 
$N(r)>0$ for $r\in(0,R]$, hence $g'(r)>0$.

Now we turn to prove 
\begin{align}\label{312}
    \frac{g'(r)}{g(r)}\ge -\a, \quad r\in (0, R].
\end{align}
Set  $v(r)=\frac{g'(r)}{g(r)}$, then $v(R)=-\a$, $v(r)>0$ for $r\in(0,R]$ and $\lim_{r\rightarrow 0^+}v(r)=+\infty$.  Using $H'(r)=-\l_1(S_r)$ and rewriting equation \eqref{odeg} as an ODE for $v$ yields
\begin{align}\label{odev}
v'+v^2+Hv+\left(\l_{2,\a}(B_R)+H'\right)=0.
\end{align}
 Suppose  \eqref{312} is not true, then there exists $r_0\in(0,R)$ such that
\begin{align*}
v'(r_0)=0,\text{\quad \quad} v''(r_0)\ge 0,\text{\quad and\quad} v(r_0)<-\a.
\end{align*}
On one hand, by differentiating \eqref{odev} in $r$, we have that at $r=r_0$ 
\begin{align}\label{alpha}
0&=v''(r_0)+H'(r_0)v(r_0)+H''(r_0) \nonumber\\
&>-\a H'(r_0) +H''(r_0).
\end{align}
On the other hand, via direct calculations, we have
\begin{align*}
\frac{H''}{H'}=&-2\coth r-2(k-1)\frac{\tanh r}{(m-k)\cosh^2 r+k-1}\\
<& -2 \coth R\\
< &\a,
\end{align*}
where in the last inequality we used assumption $\a\in[-\sigma_1(B_R), 0)$ and inequality \eqref{sig}.
In particular,  we have
$$
-\a H'(r_0)+H''(r_0)\ge 0,
$$
contradicting with \eqref{alpha}.  Hence inequality \eqref{312} comes true.
 
\emph {Proof of (2)}. Since 
\begin{align*}
\int_{\p B_R} g(r) \psi_i(\theta) \, d A=0,\quad 1\le i\le m,
\end{align*}
the functions $u_i=g(r)\psi_i(\theta)$ ($1\le i\le m$) are test functions for $\sigma_1(B_R)$. Where $\psi_i$'s are the linear coordinate functions restricted to $\mathbb{S}^{m-1}$. Therefore, by \eqref{var-stek} we get
\begin{align*}
\sum\limits_{i=1}^m\int_{B_R}|\nabla u_i|^2 \, d\mu&\ge\sigma_1(B_R)\sum\limits_{i=1}^m\int_{\p B_R}|u_i|^2 \, dA\\
&\ge-\a \sum\limits_{i=1}^m\int_{\p B_R}|u_i|^2 \, dA.
\end{align*}
Recall that
\begin{align*}
\l_{2,\a}(B_R)&=\frac{\sum\limits_{i=1}^m\int_{B_R}|\nabla u_i|^2 \, d\mu+\a \sum\limits_{i=1}^m\int_{\p B_R}|u_i|^2 \, dA}{\sum\limits_{i=1}^m\int_{B_R}u_i^2 \, d\mu},
\end{align*}
so  $\l_{2,\a}(B_R)\ge 0.$
\end{proof}

\section{A monotonicity lemma}\label{sect.4}
Relabelling the solution to \eqref{odeg} as $g_1$, we define (with a slight abuse of notation) the function $g:[0,\infty)\to[0,\infty)$ by
\begin{equation}\label{defg}
g(r) := 
\begin{cases}
g_1(r), & r\leq R,\\
g_1(R)e^{-\a(r-R)}, & r>R.
\end{cases}
\end{equation}
By definition, $g$ is continuously differentiable and  non-decreasing on $(0,\infty)$.  The following monotonicity lemma plays a key role in the proof of Theorem \ref{thm}. 

\begin{lemma}\label{monotone}
Assume that $\a \in [-\sigma_1(B_R),0]$. Define $F:[0,\infty)\to\mathbb{R}$ by
\begin{equation}\label{defH}
F(r):=g'(r)^2-H'(r) g^2(r)+2\alpha g(r) g'(r)+\alpha H(r) g^2(r),
\end{equation}
where $g(r)$ is defined in \eqref{defg}. Then $F$ is monotonically decreasing on $(0,\infty)$, and
\begin{equation}\label{F'}
F'(r)\le -\frac{7}{8}\big( \frac{m-k}{\sinh^3 r}+\frac{k-1}{\sinh^3 r\cosh^3 r}\big)g^2(R)
\end{equation}
for $r> R$.
\end{lemma}
\begin{proof}
{\bf Case 1.} $0<r\le R$. Differentiating $F$ in $r$ yields
\begin{align*}
F'(r)&=\underbrace{2g'(r)g''(r)-H''(r)g^2(r)-2H'(r)g(r)g'(r)}_{I}\\
&\quad +\underbrace{2\a g'(r)^2+2\a g(r) g''(r)+\a H'(r)g^2(r)+2\a H(r)g(r)g'(r)}_{II}.
\end{align*}
Using the ODE \eqref{odeg} of $g(r)$, we have
\begin{align*}
I&=2g'(r)g''(r)-H''(r)g^2(r)-2H'(r)g(r)g'(r)\\
&=-2H(r)g'(r)^2-H''(r)g^2(r)-4H'(r)g(r)g'(r)-2\l_{2,\a}(B_R)g(r)g'(r)\\
&=-2\big((m-1)\coth r+(k-1)\tanh r\big)g'(r)^2
+4\big(\frac{m-k}{\sinh^2 r}+\frac{k-1}{\sinh^2 r\cosh^2 r}\big)g(r)g'(r)\\
&\quad -2\big((m-k)\frac{\cosh r}{\sinh^3 r}+(k-1)\frac{\cosh^2 r+\sinh^2 r}{\sinh^3r\cosh^3 r}\big)g(r)^2-2\l_{2,\a}(B_R) g(r) g'(r).
\end{align*}
Since
\begin{align*}
&\quad 2\big(\frac{m-k}{\sinh^2 r}+\frac{k-1}{\sinh^2 r\cosh^2 r}\big)gg'-\big((m-1)\coth r+(k-1)\tanh r\big)g'(r)^2\\
&\quad -\big((m-k)\frac{\cosh r}{\sinh^3 r}+(k-1)\frac{\cosh^2 r+\sinh^2 r}{\sinh^3r\cosh^3 r}\big)g^2\\
&=-(m-k)\big(\frac{\cosh r}{\sinh r} g'(r)^2-\frac{2}{\sinh^2 r} gg'+\frac{\cosh r}{\sinh^3 r} g^2\big)\\
&\quad-(k-1)(\coth r+\tanh r)g'(r)^2+2(k-1)\big(\frac{1}{\sinh^2 r \cosh^2 r}\big)gg'\\
&\quad-(k-1)\big(\frac{\cosh^2 r+\sinh^2 r}{\sinh^3r\cosh^3 r}\big)g^2\\
&\le- (k-1)\big(\frac{\sinh^2r+\cosh^2 r}{\sinh r \cosh r}g'(r)^2-\frac{2gg'}{\sinh^2 r\cosh^2 r}+\frac{\sinh^2r+\cosh^2 r}{\sinh^3 r \cosh^3 r}g^2\big)\\
&< 0,
\end{align*}
then we have
\begin{align}\label{I}
    I\le -2\l_{2,\a}(B_R) g(r)g'(r).
\end{align}
Using \eqref{odeg} again, we estimate
\begin{align}\label{II}
II&= 2\a (g')^2+2\a g g''+\a H'g^2+2\a Hgg'\nonumber\\
&= 2\a (g')^2-\a H'g^2-2\a \l_{2,\a}(B_R)g^2\nonumber\\
&=\a\left(2(g')^2+(\frac{m-k}{\sinh^2 r} +\frac{k-1}{\sinh^2 r \cosh^2 r}) g^2-2 \l_{2,\a}(B_R) g^2\right)\nonumber\\
&\le-2 \a \l_{2,\a}(B_R) g^2.
\end{align}
Combing \eqref{I} and \eqref{II} together, we get
\begin{align*}
F'(r)&=I+II\\
&< -2\l_{2,\a}(B_R)gg'-2\a\l_{2,\a}(B_R)g^2\\
&=- 2\l_{2,\a}(B_R)g(r)\left(g'(r)+\a g(r)\right)\\
&\le 0,
\end{align*}
where in the last inequality we used $g'\geq -\a g$ on $(0,R]$ and $\lambda_{2,\a}(B_R)\ge 0$. Therefore, $F(r)$ is monotonically decreasing on $(0,R]$.

{\bf Case 2.} $r\ge R$. Recall from  \eqref{defg} that $g(r)=g(R)e^{-\a(r-R)}$, so we have
\begin{align*}
F(r)=\left(-\a^2-H'(r)+\a H(r)\right)g^2(r).
\end{align*}
Differentiating $F$ in $r$ yields
$$
F'(r)=\left(2\a^3-H''(r)+3\a H'(r)-2\a^2 H(r)\right)g^2(r).
$$
Since
\begin{align*}
&\quad H''(r)-3\a H'(r)+2\a^2 H(r)\\
&=2\frac{m-k}{\sinh^3 r}+2(k-1)\frac{\sinh^2 r+\cosh^2r}{\cosh^3 r\sinh^3 r}+3\a(\frac{m-k}{\sinh^2 r}+\frac{k-1}{\sinh^2 r\cosh^2 r})\\
&\quad+2\a^2\left((m-k)\frac{\cosh r}{\sinh r}+(k-1)\frac{\sinh^2 r+\cosh^2 r}{\sinh r \cosh r}\right)\\
&\ge \frac{m-k}{\sinh r}\left(\frac{2}{\sinh^2r}+\frac{3\a}{\sinh r}+2\a^2\right)\\
&\quad +\frac{k-1}{\cosh r \sinh r}\left(\frac{2}{\sinh^2r\cosh^2r}+\frac{3\alpha}{\sinh r\cosh r}+2\a^2\right)\\
&\ge \frac{7}{8}\big( \frac{m-k}{\sinh^3 r}+\frac{k-1}{\sinh^3 r\cosh^3 r}\big),
\end{align*}
then 
\begin{align*}
F'(r)&\le -\frac{7}{8}\big( \frac{m-k}{\sinh^3 r}+\frac{k-1}{\sinh^3 r\cosh^3 r}\big)g^2(r)\\
&\le -\frac{7}{8}\big( \frac{m-k}{\sinh^3 r}+\frac{k-1}{\sinh^3 r\cosh^3 r}\big)g^2(R),
\end{align*}
proving \eqref{F'}, and $F(r)$ is monotonically decreasing on $[R,+\infty)$.

\end{proof}

\section{Proof of Theorem \ref{thm}}\label{sect.5}
In this section, we shall prove  Theorem \ref{thm}. Before doing this, we first recall  the following centre of mass lemma, which will be used later.
\begin{lemma}\label{lmtest}
There exists a point $p$ in the closed geodesic convex hull of $\Omega$, such that
\begin{equation}\label{test}
\int_{\Omega} g(r_p(x))\frac{\exp_p^{-1}(x)}{r_p(x)} u_1(x) \,d\mu =0,
\end{equation}
where $g$ is defined in \eqref{defg}, $r_p(x)=\operatorname{dist}_g(p,x)$, $\exp_p^{-1}$ is the inverse of the exponential map $\exp_p:T_p M\to M$, and $u_1$ is a first positive eigenfunction for $\lambda_{1, 
\a}(\Omega)$.
\end{lemma}
\begin{proof}
This lemma can be proved by Brouwer’s fixed point theorem (cf. \cite[Page 635]{Wei56}). Here we give a more direct proof by following  arguments  
in  \cite[Prop. 3]{FL21} and \cite[Sect. 7.4.3]{He17}. Let 
$$
G(y)=\int_\Omega \big(\int_0^{r_y(x)} g(t)\,dt\big) u_1(x)\, d\mu_x, \quad\quad y\in M.
$$
Since $g$ is monotone increasing on $(0, \infty)$ and $u_1(x)$ is positive on $\Omega$, then $G(y)$ attains its minimum at $p$ in the convex hull of $\Omega$, hence
$$
0=\nabla G(p)=\int_{\Omega} g(r_p(x))\frac{\exp_p^{-1}(x)}{r_p(x)} u_1(x) \,d\mu, 
$$
proving the lemma.
\end{proof}
Now we  turn to prove Theorem \ref{thm}.
\begin{proof}[Proof of Theorem \ref{thm}]
From here on, we fix the point $p$ according to Lemma \ref{lmtest} so that \eqref{test} holds. Let $(r, \theta)$ denote the polar coordinates centered at $p$ and $S_r$ denote the distance sphere with respect to $p$. Then we have
\begin{align*}
\frac{\exp^{-1}_p(x)}{r_p(x)} = \left(\psi_1(\theta),\psi_2(\theta),\cdots, \psi_{m}(\theta) \right),
\end{align*}
where $\psi_i$'s are the restrictions of the linear coordinate functions on $\mathbb{S}^{m-1}$. 

For $1\le i\le m$, we define
\begin{align*}
v_i(x) := g(r_p(x)) \psi_i(\theta), 
\end{align*}
and rewrite \eqref{test} as
\begin{align*}
\int_{\Omega} v_i(x)u_1(x) \, d\mu =0.
\end{align*}
So $v_i$'s are test functions for $\l_{2, \a}(\Omega)$ and we conclude
\begin{align}\label{suml}
\l_{2, \a}(\Omega) \le  \frac{\sum\limits_{i=1}^{m} \int_{\Omega} |\nabla v_i|^2\, d\mu+\a\sum\limits_{i=1}^{m}  \int_{\p \Omega} v_i^2\, dA}{\sum\limits_{i=1}^{m}  \int_{\Omega} v_i^2\, d\mu} .
\end{align}
Using \eqref{sumpsi} and $\sum_{i=1}^{m}\psi_i^2=1$, we compute that
\begin{align}\label{sum1}
\sum_{i=1}^{m} \int_{\Omega} \left|\nabla  v_i\right|^2 \, d\mu \nonumber
&= \sum_{i=1}^{m} \int_{\Omega} \left|\nabla  \left(g(r_p)\psi_i\right)\right|^2\, d\mu \nonumber \\
&=\sum_{i=1}^{m}\int_{\Omega}\left|g'(r_p)\right|^2  \psi^2_i + g^2(r_p) |\nabla^{S_r} \psi_i|^2  \, d\mu \nonumber \\
&= \int_{\Omega} \left|g'(r_p)\right|^2-H'(r_p)g^2(r_p) \, d\mu,
\end{align}
and
\begin{equation}\label{sum2}
    \sum_{i=1}^{m} \int_{\Omega} v_i^2 \, d\mu =\sum_{i=1}^{m} \int_{\Omega} |g(r_p)|^2 \psi_i^2\, d\mu  =\int_{ \Omega} |g(r_p)|^2\, d\mu,
\end{equation}
and
\begin{equation}\label{sum3}
    \sum_{i=1}^{m} \int_{\p \Omega} v_i^2 \, dA =\sum_{i=1}^{m} \int_{\p \Omega} |g(r_p)|^2 \psi_i^2\, dA  =\int_{\p \Omega} |g(r_p)|^2\, dA.
\end{equation}
Using $|\nabla r_p | =1$ and  $\Delta r_p=H(r_p)$, we estimate
\begin{align}\label{sig2}
\int_{\partial \Omega} g^2(r_p) \, dA \nonumber &\ge \int_{\partial \Omega} g^2(r_p)\langle \nabla r_p, \nu \rangle \, dA \nonumber\\
&=\int_\Omega \operatorname{div}\left(g^2(r_p)\nabla r_p\right)\, d\mu \nonumber\\
&=\int_\Omega  (g^2)'(r_p) +g^2(r_p)\Delta r_p \, d\mu \nonumber\\
&=\int_\Omega  (g^2)'(r_p) +H(r_p)g^2(r_p)  \, d\mu,
\end{align} 
where $\nu$ is the outward unit normal to $\p\Omega$.
So substituting \eqref{sum1},\eqref{sum2} and \eqref{sum3} into  \eqref{suml} yields
\begin{align*}
\l_{2, \a}(\Omega)\le \frac{\int_{\Omega}\left( \left|g'(r_p)\right|^2-H'(r_p)g^2(r_p) +2\a g(r_p)g'(r_p)+\a H(r_p) g^2(r_p)\right)\, d\mu}{\int_{\Omega}  g^2(r_p)\, d\mu},
\end{align*}
where we used \eqref{sig2} and the fact $\a\le 0$. That is
\begin{equation}\label{est-l}
 \l_{2, \a}(\Omega)\le \frac{\int_{\Omega} F(r_p)\, d\mu}{\int_{\Omega}  g^2(r_p)\, d\mu}.
\end{equation}
Let $B_R$ be the geodesic ball of radius $R$,  having the same volume as $\Omega$'s  and centered at $p$ so that \eqref{test} holds.
Recall that $g(r)\psi_i(\theta)$ are the eigenfunctions corresponding to $\l_{2,\a}(B_R)$, so then
\begin{align*}
\l_{2,\a}(B_R)=\frac{\int_{B_R} |g'|^2(r_p)-H'(r_p)g^2(r_p)\, d\mu + \a\int_{\p B_R} g^2(r_p)\, dA}{\int_{ B_R} g^2(r_p)\, d\mu}
\end{align*}
and
\begin{align*}
\int_{\partial B_R} g^2(r_p)\, dA&=\int_{\partial B_R} \langle g^2(r_p)\nabla r_p, \nu\rangle\, dA\\
&=\int_{B_R} \operatorname{div}\left( g^2(r_p)\nabla r_p \right)\, d\mu\\
&=\int_{B_R}(g^2)'+g^2\Delta r_p  \, d\mu\\
&=\int_{B_R} (g^2)'+H(r_p) g^2 \, d\mu.
\end{align*}
Therefore we conclude 
\begin{align}\label{eqlB}
\l_{2,\a}(B_R) =\frac{\int_{B_R} F(r_p)\, d\mu}{\int_{B_R} g^2(r_p)\, d\mu} .
\end{align}
Recall that $g$ defined in \eqref{defg} is increasing, we have
\begin{align}\label{ineqg^2}
\int_{\Omega} g^2(r_p) \, d\mu &= \int_{\Omega\cap B_R} g^2(r_p)\, d\mu + \int_{\Omega\setminus B_R} g^2(r_p)\, d\mu \nonumber \\
&= \int_{B_R} g^2(r_p)\, d\mu + \int_{\Omega\setminus B_R} g^2(r_p) \, d\mu-\int_{B_R \setminus \Omega} g^2(r_p)\, d\mu \nonumber \\
&\ge \int_{B_R} g^2(r_p)\, d\mu. 
\end{align}
On the other hand, by Lemma \ref{monotone}, $F$ is monotonically decreasing, so then
\begin{align}\label{ineqF1}
\int_{\Omega} F(r_p)\, d\mu &= \int_{ B_R} F(r_p)\, d\mu + \int_{\Omega\setminus B_R} F(r_p) \, d\mu-\int_{B_R\setminus \Omega } F(r_p)\, d\mu \nonumber \\
&\le \int_{ B_R} F(r_p) \, d\mu + \int_{\Omega\setminus B_R} F(r_p)-F(R)\, d\mu \nonumber \\
&\le \int_{B_R} F(r_p)\, d\mu + \int_{B_{R_1}\setminus B_R} F(r_p)-F(R) \, d\mu,
\end{align}
where $B_{R_1}$ is the ball of radius $R_1$ centered at $p$  satisfying $$|B_{R_1}\setminus B_R|=|\Omega\setminus B_R|.$$
Using \eqref{F'}, we  estimate that
\begin{align}\label{ineqF2}
&\quad \int_{B_{R_1}\setminus B_R}F(r_p)-F(R) \, d\mu\nonumber\\
&\le  -\int_{B_{R_1}\setminus B_R} \frac{7}{8}\big( \frac{m-k}{\sinh^3 R_1}+\frac{k-1}{\sinh^3 R_1\cosh^3 R_1}\big)g^2(R)(r_p-R) \, d\mu\nonumber\\
&\le -C \big(R_1-R\big)^2,
\end{align}
here and thereafter,  $C$ denotes a constant depending on $m$, $k$ and $|\Omega|=v$, which may change from line to line.
Inequalities \eqref{ineqF1} and \eqref{ineqF2}  yields 
\begin{equation}\label{est-3}
   \int_{\Omega} F(r_p)\, d\mu\le\int_{B_R} F(r_p)\, d\mu -C |\Omega\setminus B_R|^2, 
\end{equation}
where we used $|B_{R_1}\setminus B_R|=O((R_1-R)^2)$.

From \eqref{est-l}, \eqref{eqlB}, \eqref{ineqg^2} and \eqref{est-3}, we deduce
\begin{align*}
    \l_{2,\a}(\Omega)& \le \frac{\int_{B_R} F(r_p) \, d\mu-C|\Omega\setminus B_R|^2}{\int_{B_R}  g^2(r_p)\, d\mu}\\
    &= \l_{2,\a}(B_R)-C|\Omega\setminus B_R|^2,
\end{align*}
proving inequality \eqref{1.3}.

Equality in \eqref{1.3} clearly holds if $\Omega=B$. Since we always have $\lambda_{2,\alpha}(\Omega)\leq\lambda_{2,\alpha}(B)$, equality in \eqref{1.3} implies $\mathcal{A}(\Omega)=0$,  i.e. $\Omega=B$.

Therefore, Theorem \ref{thm} is proved.
\end{proof}
\bibliographystyle{plain}
\bibliography{ref}

\end{document}